\documentclass[oneside,a4paper,12pt,notitlepage]{article}
\usepackage[left=0.8in, right=0.8in, top=1.5in, bottom=1.5in]{geometry}
\usepackage{amsmath, amssymb, amsthm, mathtools}
\usepackage{esint}
\usepackage[colorlinks=true, linkcolor=blue, citecolor=red]{hyperref}
\usepackage{color}
\newtheorem{theorem}{Theorem}[section]
\newtheorem{proposition}[theorem]{Proposition}
\newtheorem{lemma}[theorem]{Lemma}
\newtheorem{corollary}[theorem]{Corollary}
\theoremstyle{definition}

\newtheorem{remark}[theorem]{Remark}

\newcommand{\R}{\mathbb{R}}
\newcommand{\BV}{\text{BV}}

\DeclareMathOperator{\divg}{div}

\newcommand{\W}{\mathcal{W}}

\newcommand{\tr}{{\rm tr}}

\title{Anisotropic gradient rearrangement of BV functions and applications}
\author{Gloria Paoli, Yabo Yang*}
\date{}

\newcommand{\Addresses}{{
\bigskip 
  \footnotesize 
  \textit{E-mail address}, Y.~Yang (corresponding author): \texttt{Yabo\_Yang0927@outlook.com} 
  
   \medskip 

\textsc{School of Mathematics and Statistics, Beijing Institute of Technology, Beijing 100081, P. R. China}
   
     \textit{E-mail address}, G.~Paoli: \texttt{gloria.paoli@unina.it} 
   \medskip 
   
 \textsc{Dipartimento di Matematica e Applicazioni ``R. Caccioppoli'', Universit\`a degli studi di Napoli Federico II, Via Cintia, Complesso Universitario Monte S. Angelo, 80126 Napoli, Italy.}

 \par\nopagebreak 

}}

\begin{document}

\maketitle
 \begin{abstract}
In this paper, we introduce a symmetrization technique for the distributional gradient of a function of bounded variation  in the anisotropic setting. This generalizes the result obtained in the Euclidean case in \cite{AGNT24} by separating the absolutely continuous part of the anisotropic gradient from its singular part. Our main result is  an $L^1$ comparison between the function and its anisotropic symmetrization. Moreover, as an application,  we 
 derive isoperimetric inequalities for some geometric functionals related to the torsional rigidity.

\textsc{Keywords:} Anisotropy, BV functions, rearrangement of gradient \\
\textsc{MSC 2020:}  26A45, 35A23, 35B45; 35J60
\end{abstract}


\section{Introduction}

Symmetrization and rearrangement methods are a classical tool to compare solutions of PDEs defined on a general domain $\Omega\subset\R^n$
with suitably symmetric solutions defined on a ball (or, more generally, on an isoperimetric minimizer) having the same volume.
Starting from the pioneering work of Talenti \cite{Talenti76}, such techniques yield sharp comparison principles and extremal inequalities
for elliptic problems, and they are tightly connected with isoperimetric-type inequalities. The references for the study of the symmetrization of the Laplace operator is very wide, see e.g. \cite{AFLT97,lions1981solutions,lions1982generalized,T2,T4,T3} for the Dirichlet boundary conditions,  \cite{ANT, BG, MP23, MP24} for Robin boundary conditions, and \cite{CY24, CY25,Sannipoli22} for the anisotropic setting. 

The starting point of this work is the following comparison result, concerning solutions to Hamilton-Jacobi equations with Dirichlet boundary conditions, proved in \cite{giarrusso-Nunziante1984}.

\begin{theorem}[Giarrusso-Nunziante]
  Let $p\geq 1$, $f: \Omega \rightarrow \R$, $F:\R^n\rightarrow \R$ be measurable non-negative functions and let $K:[0,+\infty)\rightarrow [0,+\infty)$ be a strictly increasing  function such that
    \[
0 \leq K(|y|) \leq F(y), \quad \forall y \in \mathbb{R}^n, \quad \quad K^{-1}(f) \in L^p(\Omega)
\]
    Let $u \in W^{1,p}_{0}(\Omega)$ be a function that satisfies
    \begin{align}
\label{hamilton_jacobi}
    \begin{cases}
        F(\nabla u) = f(x) & \text{in } \Omega, \\
        u = 0 & \text{on } \partial \Omega,
    \end{cases}
\end{align}
and let $v$ be the unique decreasing spherically symmetric solution to
\begin{align}
\label{hamilton_jacoby_sym}
    \begin{cases}
        K(|\nabla v|) = f_{\#}(x) & \text{in } \Omega^{\#}, \\
        v = 0 & \text{on } \partial \Omega^{\sharp},
    \end{cases}
\end{align}
where $\Omega^{\sharp}$ is the ball centered at the origin with the same measure as $\Omega$, and $f_{\#}$ is the increasing rearrangement of $f$ (see Section \ref{sec:preliminaries} for the definition). Then, we have the following
\begin{equation*}
\|u\|_{L^1(\Omega)} \le \|v\|_{L^1(\Omega^{\sharp})}.
\end{equation*}
\end{theorem}

We refer to \cite{betta1996uniqueness, mercaldo_remark} for the rigidity results,  and to \cite{ALT89, cianchi1996lq, ferone_posteraro} for $L^q$ comparison. Moreover, we refer to \cite{gentile_amato} for the comparison of solution to Hamilton-Jacobi equations with non-zero boundary trace, and to \cite{amato_barbato} for quantitative type results.


A more delicate question concerns rearrangements of the gradient itself, and the extension of these ideas to the space $\BV$,
where the distributional derivative splits into an absolutely continuous part and a singular part.
Sharp results in this direction go back to  \cite{CF}, where  the interplay between Schwarz symmetrization and the
 structure of $Du$ for $u\in\BV(\R^n)$ is studied, proving a  P\'olya-Szeg\"o inequality for $BV$ functions. Recently, in  \cite{AGNT24}, the authors introduced a new symmetrization
procedure for $\BV$ functions which preserves the total variation while separating the absolutely continuous and singular parts
of the derivative: the absolutely continuous part is rearranged in a precise way, and the singular part is concentrated on the boundary of the
symmetrized domain. The main result  in  \cite{AGNT24} is an $L^1$-comparison between $u$ and its gradient symmetrization. 

The purpose of this paper is to extend the construction and the comparison principle of \cite{AGNT24} to the anisotropic setting.
Anisotropic surface energies naturally arise in several models (e.g.\ crystalline surface tension and interfacial energies), where the cost of a
hypersurface depends on the direction of its normal. Anisotropy is encoded by a convex, positively $1$-homogeneous function
$H:\R^n\to[0,+\infty)$, with polar function $H^\circ$. In this framework, the Euclidean perimeter is replaced by the anisotropic perimeter
\[
P_H(E)=\int_{\partial^*E} H(\nu_E)\,d\mathcal H^{n-1},
\]
and the role of the Euclidean ball is played by the Wulff ball, which minimizes $P_H$ at fixed volume.

Our main result provides an anisotropic counterpart of the gradient symmetrization in \cite{AGNT24}.
Given a non-negative function $u\in \BV_0(\Omega)$, we construct a function $u^\star$ on the Wulff
ball $\Omega^\star$  with the same volume of $\Omega$, in such a way that:
\begin{itemize}
\item the distribution of the anisotropic absolutely continuous energy density
\(H(\nabla^a u)\) matches that of \(H(\nabla u^\star)\), namely
\[
\big(H(\nabla^a u)\big)^*
=
\big(H(\nabla u^\star)\big)^*;
\]
here \(f^*\) denotes the decreasing rearrangement of a measurable function \(f\)
(see Section~\ref{prel:rearrangements});
\item the singular part of the distributional derivative $Du$ is not lost under symmetrization:
its contribution is encoded into the boundary trace of $u^\star$ on $\partial\Omega^\star$, so that
$D^s u$ appears as an explicit boundary term in the comparison problem for $u^\star$;
\item and the following $L^1$ comparison holds:
\[
\|u\|_{L^1(\Omega)}\le \|u^\star\|_{L^1(\Omega^\star)}.
\]
\end{itemize}
For the exact definition of $u^{\star}$ see \eqref{eq:ustar_problem}.
In particular, this procedure not only symmetrizes the absolutely continuous energy, but also 
keeps track of $D^s u$ through a boundary datum on $\partial\Omega^\star$, making the $\BV$ nature of
the argument fully explicit.
Therefore, our main theorem can be stated as follows.
We use the following notation:
\[
BV_{0}(\Omega) = \{u \in \BV(\mathbb{R}^n): u \equiv 0 \,\,\mathrm{in}\,\, \mathbb{R}^n\backslash\Omega\}.
\]
\begin{theorem}\label{thm:main_comparison}
Let $\Omega\subset\R^n$ be an open bounded set with finite perimeter and let $\Omega^{\star}$ be the Wulff ball centered at the origin with $|\Omega^\star|=|\Omega|$.
Assume that $u$ is a non-negative function in $\BV_0(\Omega)$ and let $u^\star$ be defined by \eqref{eq:ustar_problem}.
Then
\begin{equation}\label{eq:main_inequality}
\|u\|_{L^1(\Omega)}\le \|u^{\star}\|_{L^1(\Omega^\star)} .
\end{equation}
\end{theorem}

Finally, we illustrate how this anisotropic symmetrization principle can be employed in comparison arguments for variational problems and PDEs,
where the natural geometry is governed by $H$ and the extremal domain is a Wulff shape. In particular, functionals involving anisotropic
Dirichlet energies and boundary contributions can be handled by combining the main theorem with the interpretation of boundary terms
through the singular part of the distributional derivative of suitable zero extensions.

\medskip

\noindent\textbf{Plan of the paper.}
In Section~\ref{sec:preliminaries} we recall basic facts on the anisotropic setting, $\BV$ functions and rearrengements.
In Section~\ref{sec:main_result} we define the anisotropic gradient symmetrization and prove the $L^1$ comparison theorem.
In the final section, we discuss applications to variational problems in the anisotropic framework.

\section{Preliminaries}
\label{sec:preliminaries}

In this section, we fix the notation and recall some basic definitions and properties concerning anisotropic norms (also known as Finsler norms), functions of bounded variation in the anisotropic setting, and rearrangements.

We first introduce some basic notations. 
We will denote by $\chi_{\Omega}$ the characteristic of a set $\Omega$ and by $|\Omega|$ the measure of $\Omega$. The perimeter of a set $\Omega$ defined as $P(\Omega)$. Here $B_R(x)$ denotes the closed ball of radius $R$ centered at $x$ and $B_R$ centered at origin.

\subsection{Anisotropic norm and Wulff shape}
For this section we refer to \cite{AFLT97,CWYY}.
Throughout the paper we assume that \(H:\mathbb R^n\to[0,+\infty)\) is a norm
such that
\[
H\in C^2(\mathbb R^n\setminus\{0\}),
\]
and such that the unit ball of \(H\) is strictly convex and has \(C^2\) boundary.
The function $H$ is called anisotropic norm or Finsler norm, and, often, it represents the surface tension density.  In particular,
\begin{equation}\label{eq:H_properties}
H(\xi)>0\quad\text{for }\xi\neq0,\qquad
H(t\xi)=|t|H(\xi)
\quad\text{for every }t\in\mathbb R,\ \xi\in\mathbb R^n.
\end{equation}
Since \(H\) is a norm, we have that 
\[
H(-\xi)=H(\xi)\qquad\text{for every }\xi\in\mathbb R^n.
\]
Moreover, we have that the anisotropy $H$ satisfies
\begin{equation} \label{control}
\alpha |\xi | \le H(\xi) \le \beta |\xi|,
\end{equation}
for some positive constants $\alpha \le \beta$. The polar norm is denoted by $H^\circ: \R^n \to [0, +\infty)$  and it is defined in the following way 
\begin{equation}
    H^\circ(x) = \sup_{\xi \neq 0} \frac{\langle x, \xi \rangle}{H(\xi)},
\end{equation}
 recalling that 
\begin{equation*}
H(x) = \sup_{\xi \neq 0 } \frac{\left< x , \xi \right>}{H^{\circ}(\xi)}.
\end{equation*}
The Wulff ball, or Wulff shape,  centered at the origin with radius $R$ is defined as:
\[
    \mathcal{W}_R = \{ x \in \R^n : H^\circ(x) < R \}.
\]
 We denote the unit Wulff ball by $\mathcal{W} = \mathcal{W}_1$ and its Lebesgue measure by $\kappa_n = |\mathcal{W}|$. Note that in the Euclidean case where $H(\xi)=|\xi|$, $\mathcal{W}$ is the standard unit ball and $\kappa_n = \omega_n$.
We also recall the following properties for $H$ and $H^{\circ}$:
\begin{equation} \label{prodsc}
H_{\xi}(\xi)\cdot \xi = H(\xi), \;\;\;\; H^{\circ}_{\xi}(\xi)\cdot \xi = H^{\circ}(\xi),
\end{equation} 
\begin{equation} \label{HHcirc}
H(H^{\circ}_{\xi}(\xi))= H^{\circ}(H_{\xi}(\xi))=1 \,\,\,\,\, \forall \xi \in \mathbb{R}^n \setminus \{0\},
\end{equation} 
\begin{equation}\label{Hcircxi}
H^{\circ}(\xi)H_{\xi}(H^{\circ}_{\xi}(\xi))=H(\xi)H^{\circ}_{\xi}(H_{\xi}(\xi)) = \xi \,\,\, \forall \xi \in \mathbb{R}^n \setminus \{0\}.
\end{equation}

\subsection{Anisotropic total variation and perimeter}
For this section we refer to \cite{amar,ambrosio2000functions, fusco_cianchi,maggi2012sets}.
Let $\Omega \subset \R^n$ be an open set. We say that a function $u \in L^1(\Omega)$ is a function of bounded variation, denoted by $u \in \BV(\Omega)$, if its distributional derivative is a vector-valued Radon measure with finite total variation. The space $\BV(\Omega)$ of functions of bounded variation in $\Omega$ is a Banach space when endowed with the norm $\|u\|_{BV(\Omega)}=\|u\|_{L^1(\Omega)}+|Du|(\Omega)$. 
The anisotropic total variation of $u$ in $\Omega$ with respect to $H$ is defined as:
\begin{equation} \label{eq:aniso_TV_def}
    |Du|_H(\Omega) = \sup \left\{ \int_\Omega u \divg \zeta \, dx : \zeta \in C^1_c(\Omega, \R^n), \, H^\circ(\zeta(x)) \leq 1 \text{ for all } x \in \Omega \right\}.
\end{equation}
By the Riesz Representation Theorem and the Lebesgue decomposition of the measure $Du$, this total variation admits an integral representation. Specifically, if $Du = \nu_u |Du|$ is the polar decomposition where $\nu_u$ is the Radon-Nikodym derivative, then:
\begin{equation} \label{eq:aniso_TV_integral}
    |Du|_H(\Omega) = \int_\Omega H(\nu_u) \, d|Du|.
\end{equation}
Using the decomposition
$|Du| = |D^au| + |D^su|$, with $D^au\ll \mathcal{L}^{n}$ and  $D^su \perp \mathcal{L}^{n}$, this can be written explicitly as:
\[
    |Du|_H(\Omega) = \int_\Omega H(\nabla^a u) \, dx + \int_\Omega H\left(\frac{d D^s u}{d |D^s u|}\right) \, d|D^s u|.
\]
For a measurable set $E \subset \R^n$, the anisotropic perimeter of $E$ in $\Omega$ is defined as the anisotropic total variation of its characteristic function:
\[
    P_H(E; \Omega) = |D\chi_E|_H(\Omega) = \int_{\partial^* E \cap \Omega} H(\nu_E) \, d\mathcal{H}^{n-1},
\]
where $\partial^* E$ is the reduced boundary of $E$ (see \cite{evans2015measure}) and $\nu_E$ is the generalized outer normal.
If $E$ is an open set with Lipschitz boundary, then the outer unit normal exists almost everywhere, $\partial^* E$ coincides with $\partial E$ (up to $\mathcal{H}^{n-1}$-negligible sets), and we recover the integral formulation:
\begin{equation}
P_H(E,\Omega) = \int_{\partial E \cap \Omega} H(\nu )\,d\mathcal{H}^{n-1}.
\end{equation}
By the properties of $H$, it is well known that there exist constants $\alpha, \beta > 0$ such that:
\begin{equation*}
\alpha P(E) \le P_H(E)\le \beta P(E).
\end{equation*}

A fundamental tool in our analysis is the anisotropic Coarea Formula (Fleming-Rishel formula). For any $u \in \BV(\Omega)$, it holds:
\begin{equation} \label{eq:coarea}
    |Du|_H(\Omega) = \int_{-\infty}^{+\infty} P_H(\{u > t\}, \Omega) \, dt.
\end{equation}
We remark that if $u \in W^{1,1}(\Omega)$, the measure $Du$ is absolutely continuous with respect to the Lebesgue measure ($D^s u = 0$). In this case, the total variation reduces to $\int_\Omega H(\nabla u) dx$, and the coarea formula takes the  form (see \cite{ALT90}):
\begin{equation} \label{AFLTAB}
    -\frac{d}{dt}\int_{\{u>t\}} H(\nabla u)\,dx = \int_{\partial^*\{u>t\}\cap \Omega} \frac{H(\nabla u)}{|\nabla u|}\,d\mathcal{H}^{n-1}.
\end{equation}

Finally, we recall the anisotropic isoperimetric inequality \cite{busemann1949isoperimetric, dacorogna1992wulff, fonseca1991uniqueness}: for any set $E$ of finite perimeter,
\begin{equation} \label{eq:isoperimetric}
    P_H(E) \geq n \kappa_n^{1/n} |E|^{1 - 1/n}.
\end{equation}
Equality holds if and only if $E$ is equivalent to a Wulff ball (up to translation and dilation).
\begin{remark}

We emphasize that due to the properties of $H$ (see in particular \eqref{eq:H_properties}), there exist constants $\alpha, \beta > 0$ such that $\alpha|\xi| \le H(\xi) \le \beta|\xi|$ for all $\xi \in \R^n$. Consequently, the anisotropic total variation $|Du|_H(\Omega)$ is finite if and only if the standard Euclidean total variation $|Du|(\Omega)$ is finite. Therefore, the space $\BV(\Omega)$ remains the same as in the classical setting; only the measure of the total variation changes.
\end{remark}

\subsection{Rearrangements}\label{prel:rearrangements}

Let $u: \Omega \to \R$ be a measurable function. The distribution function of $u$ is defined as:
\[
    \mu(t) = |\{ x \in \Omega : |u(x)| > t \}|, \quad t \geq 0.
\]
The decreasing rearrangement $u^*: [0, |\Omega|] \to [0, +\infty)$ is defined as:
\[
    u^*(s) = \inf \{ t \geq 0 : \mu(t) \leq s \}.
\]
We define the anisotropic decreasing symmetrization of $u$, denoted by $u^\sharp$, as the function defined on the symmetrized Wulff ball $\Omega^\star = \mathcal{W}_R$ (where $R$ is chosen such that $|\Omega^\star| = |\Omega|$). It is given by:
\begin{equation} \label{eq:ustar_def}
    u^\sharp(x) = u^*(\kappa_n (H^\circ(x))^n), \quad x \in \Omega^\star.
\end{equation}
The function $u^\sharp$ is radial with respect to the anisotropic distance $H^\circ$, decreasing with respect to $H^\circ(x)$, and is equimeasurable with $u$. 
Moreover, we define the increasing rearrangements of $u$ as 
$$u_*(s)=u^*(|\Omega|-s), $$
and the anisotropic increasing symmetrization as 
$$ u_\sharp(x)=u_*(\kappa_n (H^\circ(x))^n), \quad x \in \Omega^\star.$$
Consequently, for any $p \in [1, \infty]$, we have: 
\[
    \|u\|_{L^p(\Omega)} = \|u^*\|_{L^p(0, |\Omega|)} = \|u^\sharp\|_{L^p(\Omega^\sharp)}=\|u_*\|_{L^p(0, |\Omega|)} = \|u_\sharp\|_{L^p(\Omega^\sharp)}.
\]

A key property we will use is the Hardy-Littlewood inequality: for any non-negative functions $f, g \in L^1(\Omega)$,
\[
    \int_\Omega |f(x) g(x)| \, dx \leq \int_0^{|\Omega|} f^*(s) g^*(s) \, ds=\int_{\Omega^\star} f^\sharp(x)g^\sharp(x)dx. 
\]
In particular, choosing $g(x) = \chi_{\{u > t\}}(x)$, we obtain:
\begin{equation} \label{eq:HL_levelsets}
    \int_{\{u > t\}} f(x) \, dx \leq \int_0^{\mu(t)} f^*(s) \, ds.
\end{equation}
Moreover, we have
\[
    \int_{\Omega^\star }f^\sharp(x) g_\sharp(x) \, dx = \int_0^{|\Omega|} f^*(s) g_*(s) \, ds\leq\int_{\Omega} |f(x)g(x)|dx, 
\]

Regarding the norm of the gradient, a generalized version of the  well known P\'olya-Szeg\"o inequality holds and it states (see for instance \cite{AFTL97}), if $w \in W^{1,p}_0(\Omega)$ for $p\ge 1$, then we have that
\begin{equation*}
    \int_{\Omega}H(\nabla u)^p\,dx \ge \int_{\Omega^{\star}} H(\nabla u^{*})^p\,dx.
\end{equation*}
where $\Omega^\star$ is the Wulff Shape such that $|\Omega^\star|=|\Omega|$.
The equality case, proved in the Euclidean case by \cite{Brothers1988}, and an alternative new proof was studied in \cite{ferone2003minimal}.

\section{Main Result}
\label{sec:main_result}

In this section, we construct the anisotropic gradient rearrangement and prove the
main $L^1$ comparison theorem. The argument follows the strategy of
\cite{AGNT24}, with the Euclidean perimeter replaced by the anisotropic perimeter
and the centered ball replaced by the Wulff ball.

\subsection{Technical lemmas}
\label{subsec:lemmas}

Let $\Omega\subset \R^n$ be an open, bounded set with finite perimeter, and let
$u\in \BV_0(\Omega)$ be non-negative. For every $s\ge 0$, define the truncation
\begin{equation}\label{eq:truncation_def}
    v_s(x):=(u(x)-u^*(s))_+ .
\end{equation}
Since $u\in \BV_0(\Omega)$, the function $v_s$ belongs to $\BV_0(\Omega)$ for
a.e. $s\ge 0$. We set
\begin{equation}\label{eq:G_def}
    G(s):=|Dv_s|_H(\R^n)=|D^a v_s|_H(\R^n)+|D^s v_s|_H(\R^n).
\end{equation}
The function $G$ is increasing on $[0,+\infty)$, hence it induces a positive Radon
measure $\sigma_H$ on $[0,+\infty)$ such that
\begin{equation}\label{def:sigma}
    G(s)=\int_{(0,s]} d\sigma_H(\tau),
    \qquad \forall s\in[0,+\infty).
\end{equation}

The next lemma is the anisotropic analogue of the corresponding formula in
\cite{AGNT24}.

\begin{lemma}\label{lem:G_derivative}
Let $u$ be a non-negative function in $\BV_0(\Omega)$. Then, for a.e. $s\in[0,+\infty)$,
\begin{equation}\label{eq:G_integral}
    G(s)=\int_{u^*(s)}^{+\infty} P_H(\{u>t\})\,dt.
\end{equation}
\end{lemma}

\begin{proof}
Applying the anisotropic coarea formula \eqref{eq:coarea} to $v_s$, we get
\[
|Dv_s|_H(\R^n)=\int_{-\infty}^{+\infty} P_H(\{v_s>\tau\})\,d\tau .
\]
Since $\{v_s>\tau\}=\{u>u^*(s)+\tau\}$ for every $\tau>0$, by the change of
variable $t=u^*(s)+\tau$ we obtain
\[
G(s)=\int_0^{+\infty} P_H(\{u>u^*(s)+\tau\})\,d\tau
=\int_{u^*(s)}^{+\infty} P_H(\{u>t\})\,dt,
\]
which is exactly \eqref{eq:G_integral}.
\end{proof}

We need an approximation lemma to handle the singular part of the gradient. This result, was originally proved in \cite{alvino1978sulle} in the Euclidean setting, extends naturally to our context,  ensuring we can approximate $L^1$ functions while preserving their rearrangement.
\begin{lemma}\label{lem:approximation}
Let \(u\in \BV_0(\Omega)\) be non-negative and let \(g\in L^1(\Omega)\). Define
\[
G_g(s):=\int_{\{u>u^*(s)\}} g(x)\,dx,
\qquad s\in[0,|\Omega|].
\]
Assume that \(G_g\) is absolutely continuous and write
\[
G_g(s)=\int_0^s F_g(\tau)\,d\tau .
\]
Then there exists a sequence \(\{g_k\}_{k\in\mathbb N}\subset L^1(0,|\Omega|)\) such that
\[
g_k^*=g^*
\qquad\text{for every }k,
\]
and
\[
\lim_{k\to\infty}
\int_0^{|\Omega|}g_k(\tau)\varphi(\tau)\,d\tau
=
\int_0^{|\Omega|}F_g(\tau)\varphi(\tau)\,d\tau
\qquad
\forall\,\varphi\in \BV([0,|\Omega|)).
\]
\end{lemma}

We now split the anisotropic variation of the truncation into its absolutely
continuous and singular parts. For $s\in[0,|\Omega|]$, set
\[
G_{H,1}(s):=|D^a v_s|_H(\R^n),
\qquad
G_{H,2}(s):=|D^s v_s|_H(\R^n).
\]
Then
\[
G(s)=G_{H,1}(s)+G_{H,2}(s).
\]
Moreover, by the chain rule for BV truncations,
\[
D^a v_s=\nabla^a u\,\chi_{\{u>u^*(s)\}}\,\mathcal L^n,
\]
hence
\begin{equation}\label{eq:G1_formula}
G_{H,1}(s)=\int_{\{u>u^*(s)\}} H(\nabla^a u)\,dx.
\end{equation}
Arguing exactly as in \cite[Corollary 2.5]{AGNT24}, the map $G_{H,1}$ is increasing
and absolutely continuous on $[0,|\Omega|]$. Therefore there exists a function
$F_{H,1}\in L^1(0,|\Omega|)$ such that
\begin{equation}\label{eq:G1_density}
G_{H,1}(s)=\int_0^s F_{H,1}(\tau)\,d\tau
\qquad \forall s\in[0,|\Omega|].
\end{equation}
On the other hand, $G_{H,2}$ is increasing, hence there exists a positive Radon
measure $\sigma_{H,2}$ on $[0,|\Omega|]$ such that
\begin{equation}\label{eq:G2_measure}
G_{H,2}(s)=\int_{(0,s]} d\sigma_{H,2}(\tau)
\qquad \forall s\in[0,|\Omega|].
\end{equation}
Combining \eqref{def:sigma}, \eqref{eq:G1_density}, and \eqref{eq:G2_measure}, we get
\begin{equation}\label{eq:sigma_decomposition}
d\sigma_H = F_{H,1}(s)\,ds + d\sigma_{H,2}(s)
\qquad \text{on }[0,|\Omega|].
\end{equation}
Finally, since $u^*(|\Omega|)=0$, we have
\begin{equation}\label{eq:sigma2_total_mass}
\int_{(0,|\Omega|]} d\sigma_{H,2}(\tau)
=
G_{H,2}(|\Omega|)
=
|D^s u|_H(\mathbb{R}^n).
\end{equation}

\subsection{Definition and explicit computation of the gradient rearrangement}
\label{subsec:existence}

Let $\Omega^\star$ be the centered Wulff ball such that $|\Omega^\star|=|\Omega|$.
Set
\[
g:=H(\nabla^a u)\in L^1(\Omega),
\]
and let $g_*$ be the increasing rearrangement of $g$ on $[0,|\Omega|]$. We define
the increasing anisotropic symmetrization of $g$ on $\Omega^\star$ by
\[
g_\sharp(x):=g_*\!\bigl(\kappa_n(H^\circ(x))^n\bigr),
\qquad x\in\Omega^\star .
\]
Notice that $g_\sharp$ is $H^\circ$-radial and nondecreasing with respect to
$H^\circ(x)$.

We look for a function
\[
u^\star \in \BV_0(\Omega^\star)\cap W^{1,1}(\Omega^\star)\cap L^\infty(\Omega^\star),
\]
such that
\begin{equation}\label{eq:ustar_problem}
\begin{cases}
H(\nabla u^\star)(x)=g_\sharp(x) & \text{for a.e. }x\in\Omega^\star,\\[1ex]
u^\star \text{ is } H^\circ\text{-radial and nonincreasing},\\[1ex]
u^\star(x)=c_0 & \text{for }x\in\partial\Omega^\star,
\end{cases}
\end{equation}
where
\begin{equation}\label{eq:boundary_constant}
c_0:=\frac{|D^s u|_H(\mathbb{R}^n)}{P_H(\Omega^\star)}.
\end{equation}

\begin{proposition}\label{prop:existence}
There exists a unique function $u^\star$ satisfying \eqref{eq:ustar_problem}.
Moreover, its decreasing rearrangement is given by
\begin{equation}\label{eq:ustar_rearrangement_explicit}
(u^\star)^*(s)
=
\int_s^{|\Omega|}
\frac{g_*(t)}{n\kappa_n^{1/n} t^{1-\frac1n}}\,dt
+
\frac{|D^s u|_H(\mathbb{R}^n)}{n\kappa_n^{1/n}|\Omega|^{1-\frac1n}},
\qquad s\in[0,|\Omega|].
\end{equation}
\end{proposition}

\begin{proof}
Write $\Omega^\star=W_R$, where $R>0$ is determined by $|\Omega|=\kappa_n R^n$.
We look for a solution of the form
\[
u^\star(x)=\phi(H^\circ(x)),
\]
where $\phi:[0,R]\to\R$ is nonincreasing.

Let $r=H^\circ(x)$. By the chain rule,
\[
\nabla u^\star(x)=\phi'(r)\nabla H^\circ(x),
\]
hence
\[
H(\nabla u^\star(x))
=
H\bigl(\phi'(r)\nabla H^\circ(x)\bigr)
=
|\phi'(r)|\,H(\nabla H^\circ(x)).
\]
Using the identity $H(\nabla H^\circ(x))=1$ for a.e. $x\neq 0$, and the fact that
$\phi$ is nonincreasing, we obtain
\[
H(\nabla u^\star(x))=-\phi'(r).
\]
Therefore the condition $H(\nabla u^\star)=g_\sharp$ becomes
\[
-\phi'(r)=g_\sharp(x)=g_*(\kappa_n r^n).
\]
Integrating between $r$ and $R$ gives
\[
\phi(r)-\phi(R)=\int_r^R g_*(\kappa_n\rho^n)\,d\rho.
\]
Now set $t=\kappa_n\rho^n$. Since
\[
dt=n\kappa_n\rho^{n-1}\,d\rho,
\qquad
d\rho=\frac{dt}{n\kappa_n^{1/n}t^{1-\frac1n}},
\]
we get
\[
\phi(r)-\phi(R)
=
\int_{\kappa_n r^n}^{|\Omega|}
\frac{g_*(t)}{n\kappa_n^{1/n}t^{1-\frac1n}}\,dt.
\]
Using the boundary condition $\phi(R)=c_0$, we conclude that
\[
u^\star(x)
=
\int_{\kappa_n(H^\circ(x))^n}^{|\Omega|}
\frac{g_*(t)}{n\kappa_n^{1/n}t^{1-\frac1n}}\,dt
+
c_0.
\]
Since $u^\star$ is $H^\circ$-radial and nonincreasing, its decreasing rearrangement is
obtained by setting $s=\kappa_n(H^\circ(x))^n$, and therefore
\[
(u^\star)^*(s)
=
\int_s^{|\Omega|}
\frac{g_*(t)}{n\kappa_n^{1/n}t^{1-\frac1n}}\,dt
+
c_0.
\]
Finally, since
\[
P_H(\Omega^\star)=P_H(W_R)=n\kappa_n R^{n-1}
=
n\kappa_n^{1/n}|\Omega|^{1-\frac1n},
\]
the expression for $c_0$ in \eqref{eq:boundary_constant} becomes
\[
c_0
=
\frac{|D^s u|_H(\mathbb{R}^n)}{n\kappa_n^{1/n}|\Omega|^{1-\frac1n}},
\]
and \eqref{eq:ustar_rearrangement_explicit} follows. Uniqueness is immediate from
the explicit formula.
\end{proof}

\subsection{Comparison theorem}
\label{subsec:comparison}

We can now prove the main comparison result.

\begin{theorem}\label{thm:main_comparison}
Let $\Omega\subset\R^n$ be a bounded open set with finite perimeter, and let
$\Omega^\star$ be the centered Wulff ball such that $|\Omega^\star|=|\Omega|$.
Assume that $u$ is a non-negative function in $\BV_0(\Omega)$. If $u^\star$ is the
anisotropic gradient rearrangement defined by \eqref{eq:ustar_problem}, then
\begin{equation}\label{eq:main_inequality}
\|u\|_{L^1(\Omega)}\le \|u^\star\|_{L^1(\Omega^\star)}.
\end{equation}
\end{theorem}

\begin{proof}
Let $u^*$ be the decreasing rearrangement of $u$. By Lemma~\ref{lem:G_derivative}, for a.e.
$s\in[0,|\Omega|]$,
\[
G(s)=\int_{u^*(s)}^{+\infty} P_H(\{u>t\})\,dt .
\]
 Let
$0\le s_1<s_2\le |\Omega|$. Since \(G\) is increasing and
\[
G(s)=\sigma_H((0,s]),
\]
we have
\[
\sigma_H((s_1,s_2))
=
G(s_2^-)-G(s_1),
\]
where \(G(s_2^-):=\lim_{s\to s_2^-}G(s)\). Using the representation of \(G\)
and passing to the left limit, we obtain
\begin{align}
\sigma_H((s_1,s_2))
&=
\lim_{s\to s_2^-}
\left[
G(s)-G(s_1)
\right] \notag\\
&=
\lim_{s\to s_2^-}
\int_{u^*(s)}^{u^*(s_1)} P_H(\{u>t\})\,dt . \label{eq:sigma_interval_left}
\end{align}
By the anisotropic isoperimetric inequality \eqref{eq:isoperimetric}
\begin{align}
\sigma_H((s_1,s_2))
&\ge
\lim_{s\to s_2^-}
\int_{u^*(s)}^{u^*(s_1)}
n\kappa_n^{1/n}\mu(t)^{1-\frac1n}\,dt .
\label{eq:interval_measure_lower}
\end{align}
Define now 
\[
\mathcal K(\tau):=
\int_\tau^{+\infty}
n\kappa_n^{1/n}\mu(t)^{1-\frac1n}\,dt,
\qquad \tau\ge 0.
\]
The function \(\mathcal K\) is finite and Lipschitz continuous. Indeed, by the
anisotropic isoperimetric inequality and the anisotropic coarea formula,
\[
\int_0^{+\infty}
n\kappa_n^{1/n}\mu(t)^{1-\frac1n}\,dt
\le
\int_0^{+\infty} P_H(\{u>t\})\,dt
=
|Du|_H(\mathbb R^n)<+\infty,
\]
and the integrand is bounded by
\(n\kappa_n^{1/n}|\Omega|^{1-\frac1n}\).
With this notation, \eqref{eq:interval_measure_lower} gives
\[
\sigma_H((s_1,s_2))
\ge
D[\mathcal K(u^*)]((s_1,s_2)).
\]
By the regularity of Radon measures, this implies
\begin{equation}\label{eq:measure_domination}
\sigma_H(E)\ge D[\mathcal K(u^*)](E)
\qquad
\text{for every Borel set }E\subset(0,|\Omega|).
\end{equation}

We now consider \(D[\mathcal K(u^*)]\). Since \(u^*\) is monotone decreasing,
\(-Du^*\) is a positive measure. By the chain rule for one-dimensional BV
functions,
\begin{equation}\label{eq:chain_rule_K}
D[\mathcal K(u^*)]
=
n\kappa_n^{1/n}s^{1-\frac1n}\,(-Du^*)
\qquad\text{ on }(0,|\Omega|).
\end{equation}
 Indeed, on the
absolutely continuous and Cantor parts, one uses
\[
\mu(u^*(s))=s
\qquad\text{for }|Du^*|\text{-a.e. }s.
\]
At a jump point \(s\in J_{u^*}\), one has
\[
\mu(t)=s
\qquad
\text{for }t\in\big((u^*)_+(s),(u^*)_-(s)\big).
\]
Hence
\begin{align*}
D[\mathcal K(u^*)](\{s\})
&=
\mathcal K((u^*)_+(s))-\mathcal K((u^*)_-(s))\\
&=
\int_{(u^*)_+(s)}^{(u^*)_-(s)}
n\kappa_n^{1/n}\mu(t)^{1-\frac1n}\,dt\\
&=
n\kappa_n^{1/n}s^{1-\frac1n}
\big((u^*)_-(s)-(u^*)_+(s)\big)\\
&=
n\kappa_n^{1/n}s^{1-\frac1n}(-Du^*)(\{s\}).
\end{align*}
This proves \eqref{eq:chain_rule_K}.

Combining \eqref{eq:measure_domination} and \eqref{eq:chain_rule_K}, we get
\[
n\kappa_n^{1/n}s^{1-\frac1n}\,(-Du^*)
\le
d\sigma_H
\qquad\text{on }(0,|\Omega|).
\]
Since \(u^*(|\Omega|)=0\), it follows that, for a.e. \(s\in(0,|\Omega|)\),
\begin{equation}\label{eq:u_pointwise_estimate}
u^*(s)
\le
\int_{(s,|\Omega|]}
\frac{d\sigma_H(\tau)}
{n\kappa_n^{1/n}\tau^{1-\frac1n}}.
\end{equation}
Integrating \eqref{eq:u_pointwise_estimate} over $[0,|\Omega|]$ and using Fubini's
theorem, we get
\begin{align}
\|u\|_{L^1(\Omega)}
&=
\int_0^{|\Omega|} u^*(s)\,ds \notag\\
&\le
\int_0^{|\Omega|}
\left(
\int_s^{|\Omega|}
\frac{d\sigma_H(\tau)}{n\kappa_n^{1/n}\tau^{1-\frac1n}}
\right)ds \notag\\
&=
\frac1{n\kappa_n^{1/n}}
\int_0^{|\Omega|}\tau^{\frac1n}\,d\sigma_H(\tau). \label{eq:L1_after_fubini}
\end{align}
Using the decomposition \eqref{eq:sigma_decomposition}, this becomes
\begin{equation}\label{eq:L1_split}
\|u\|_{L^1(\Omega)}
\le
\frac1{n\kappa_n^{1/n}}
\left[
\int_0^{|\Omega|}\tau^{\frac1n}F_{H,1}(\tau)\,d\tau
+
\int_{(0,|\Omega|]}\tau^{\frac1n}\,d\sigma_{H,2}(\tau)
\right].
\end{equation}

We now estimate the first term. Set
\[
g:=H(\nabla^a u).
\]
By  Lemma \ref{lem:approximation}, applied to \(g\) with
respect to the level sets of \(u\), there exists a sequence \(\{g_k\}\subset L^1(0,|\Omega|)\)
such that
\[
g_k^*=g^*
\]
and
\[
\lim_{k\to\infty}
\int_0^{|\Omega|} g_k(\tau)\varphi(\tau)\,d\tau
=
\int_0^{|\Omega|}F_{H,1}(\tau)\varphi(\tau)\,d\tau
\qquad
\forall\varphi\in BV([0,|\Omega|)).
\]
Choosing \(\varphi(\tau)=\tau^{1/n}\) and using the Hardy--Littlewood inequality,
we obtain
\[
\int_0^{|\Omega|}\tau^{1/n}F_{H,1}(\tau)\,d\tau
\le
\int_0^{|\Omega|}\tau^{1/n}g_*(\tau)\,d\tau
=
\int_0^{|\Omega|}\tau^{1/n}\big(H(\nabla^a u)\big)_*(\tau)\,d\tau .
\]
For the second term, since $\tau^{1/n}\le |\Omega|^{1/n}$ on $[0,|\Omega|]$, by
\eqref{eq:sigma2_total_mass} we have
\begin{equation}\label{eq:sigma2_estimate}
\int_{(0,|\Omega|]}\tau^{\frac1n}\,d\sigma_{H,2}(\tau)
\le
|\Omega|^{\frac1n}|D^s u|_H(\mathbb{R}^n).
\end{equation}
Combining this with \eqref{eq:L1_split}, and \eqref{eq:sigma2_estimate},
we obtain
\begin{equation}\label{eq:L1_upper_bound}
\|u\|_{L^1(\Omega)}
\le
\frac1{n\kappa_n^{1/n}}
\int_0^{|\Omega|}\tau^{\frac1n} g_*(\tau)\,d\tau
+
\frac{|\Omega|^{1/n}}{n\kappa_n^{1/n}}|D^s u|_H(\mathbb{R}^n).
\end{equation}

It remains to compute the $L^1$ norm of $u^\star$. By
\eqref{eq:ustar_rearrangement_explicit},
\[
(u^\star)^*(s)
=
\int_s^{|\Omega|}
\frac{g_*(t)}{n\kappa_n^{1/n}t^{1-\frac1n}}\,dt
+
\frac{|D^s u|_H(\mathbb{R}^n)}{n\kappa_n^{1/n}|\Omega|^{1-\frac1n}}.
\]
Integrating over $[0,|\Omega|]$ and using Fubini's theorem once again, we get
\begin{align*}
\|u^\star\|_{L^1(\Omega^\star)}
&=
\int_0^{|\Omega|}(u^\star)^*(s)\,ds \\
&=
\frac1{n\kappa_n^{1/n}}
\int_0^{|\Omega|} t^{\frac1n} g_*(t)\,dt
+
\frac{|\Omega|^{1/n}}{n\kappa_n^{1/n}}|D^s u|_H(\mathbb{R}^n).
\end{align*}
Comparing this identity with \eqref{eq:L1_upper_bound}, we conclude that
\[
\|u\|_{L^1(\Omega)}\le \|u^\star\|_{L^1(\Omega^\star)}.
\]
This proves \eqref{eq:main_inequality}.
\end{proof}

\begin{remark}
It is worth noting that the symmetrization preserves the anisotropic total variation:
\[
|Du^\star|_H(\mathbb R^n)
=
\int_{\Omega^\star} H(\nabla u^\star)\,dx
+
c_0 P_H(\Omega^\star).
\]
By construction,
\[
\int_{\Omega^\star} H(\nabla u^\star)\,dx
=
\int_\Omega H(\nabla^a u)\,dx,
\]
and, since
\[
c_0=\frac{|D^s u|_H(\mathbb R^n)}{P_H(\Omega^\star)},
\]
we obtain
\[
|Du^\star|_H(\mathbb R^n)
=
\int_\Omega H(\nabla^a u)\,dx
+
|D^s u|_H(\mathbb R^n)
=
|Du|_H(\mathbb R^n).
\]
\end{remark}

\section {Applications: Two anisotropic versions of the torsional rigidity}
\label{sec:torsion_aniso}


In this section we prove Saint--Venant type inequalities for two functionals related to torsional rigidity, in analogy to  the Euclidean argument in \cite{AGNT24}.
The key point is that, by extending Sobolev functions by $0$ outside $\Omega$, the boundary contribution becomes the \emph{jump part}
of the anisotropic total variation, and hence it is handled by the main comparison theorem.
We set the anisotropic boundary measure
\[
d\mathcal H^{n-1}_H := H(\nu_\Omega)\,d\mathcal H^{n-1}\qquad\text{on }\partial\Omega.
\]
\subsection{Penalized torsional rigidity}

Fix $\Lambda>0$ and define, for $\psi\in H_0^1(\Omega)$,
\begin{equation}\label{eq:F_Lambda_aniso}
\mathcal F^H_{\Lambda}(\psi)
:=\frac12\int_{\Omega} H(\nabla \psi)^2\,dx-\int_{\Omega}\psi\,dx
+\Lambda\,\big|\{H(\nabla\psi)\neq 0\}\big|.
\end{equation}
The associated (penalized) torsional rigidity is
\begin{equation}\label{eq:TF_aniso_def}
T_{\mathcal F}^H(\Omega,\Lambda):=-\inf_{\psi\in H_0^1(\Omega)}\mathcal F^H_{\Lambda}(\psi).
\end{equation}
As in the Euclidean case, the infimum can be searched among non-negative functions since
$\mathcal F^H_\Lambda(|\psi|)\le \mathcal F^H_\Lambda(\psi)$. 

 The main difficulty lies in proving the existence, due to the lack of lower semicontinuity. We overcome this difficulty establishing existence in the case when $\Omega$ is a Wulff ball.  As far as the uniqueness, in order to conclude as in \cite{AGNT24}, one would need the analogous of the rigidity results proved in \cite{mercaldo_remark} in the anisotropic setting, that we do not have (see next remark). 

\begin{proposition}[Penalized torsion on a Wulff ball: existence of a symmetric minimizer]
\label{prop:TF_aniso_wulff_exist_unique}
Let $H:\R^n\to[0,+\infty)$ be a norm and let $H^\circ$ be its polar norm.
Fix $\Lambda>0$ and $R>0$, and set
\[
\W_R:=\{x\in\R^n:\ H^\circ(x)<R\}.
\]
Consider, for $\psi\in H_0^1(\W_R)$,
\[
\mathcal F^H_{\Lambda}(\psi)
=\frac12\int_{\W_R} H(\nabla \psi)^2\,dx-\int_{\W_R}\psi\,dx
+\Lambda\,\big|\{H(\nabla\psi)\neq 0\}\big|.
\]
Then $\mathcal F^H_\Lambda$ admits a minimizer in $H_0^1(\W_R)$.

Moreover, there exists a nonnegative minimizer $v\in H_0^1(\W_R)$ such that
\[
v(x)=\varphi(H^\circ(x))\qquad\text{for a.e. }x\in\W_R,
\]
for a suitable function $\varphi:[0,R]\to[0,+\infty)$, and
\[
H(\nabla v)(x)=g(H^\circ(x))\qquad\text{for a.e. }x\in\W_R,
\]
for a suitable nondecreasing function $g:[0,R]\to[0,+\infty)$.
In particular, there exists $r\in[0,R]$ such that, up to null sets,
\[
\{H(\nabla v)=0\}=\W_r,
\qquad
\{H(\nabla v)\neq 0\}=\W_R\setminus\W_r.
\]
\end{proposition}

\begin{proof}
We divide the proof into three steps.

\medskip
\noindent
\textit{Step 1: boundedness from below.}
Since $H$ is a norm, there exist constants $0<\alpha\le \beta$ such that
\[
\alpha |\xi|\le H(\xi)\le \beta |\xi|
\qquad\text{for every }\xi\in\R^n.
\]
Let $C_P=C_P(\W_R)$ be the Poincar\'e constant on $H_0^1(\W_R)$, namely
\[
\|\psi\|_{L^2(\W_R)}\le C_P \|\nabla\psi\|_{L^2(\W_R)}
\qquad\text{for every }\psi\in H_0^1(\W_R).
\]
Hence
\[
\|\psi\|_{L^2(\W_R)}
\le \frac{C_P}{\alpha}\|H(\nabla\psi)\|_{L^2(\W_R)}.
\]
Using H\"older and Young inequalities, we obtain
\begin{align*}
\int_{\W_R}\psi\,dx
&\le |\W_R|^{1/2}\|\psi\|_{L^2(\W_R)}\\
&\le \frac{|\W_R|^{1/2}C_P}{\alpha}\|H(\nabla\psi)\|_{L^2(\W_R)}\\
&\le \frac14\int_{\W_R}H(\nabla\psi)^2\,dx + C,
\end{align*}
for a suitable constant $C=C(n,\alpha,R)>0$.
Therefore
\begin{align*}
\mathcal F^H_\Lambda(\psi)
&=\frac12\int_{\W_R}H(\nabla\psi)^2\,dx-\int_{\W_R}\psi\,dx
+\Lambda |\{H(\nabla\psi)\neq 0\}|\\
&\ge \frac14\int_{\W_R}H(\nabla\psi)^2\,dx - C
\ge -C.
\end{align*}
Thus $\inf_{H_0^1(\W_R)}\mathcal F^H_\Lambda>-\infty$.

\medskip
\noindent
\textit{Step 2: existence of a minimizer.}
Let
\[
m:=\inf_{\psi\in H_0^1(\W_R)} \mathcal F^H_\Lambda(\psi),
\]
and let $\{\psi_k\}\subset H_0^1(\W_R)$ be a minimizing sequence.
Replacing $\psi_k$ with $|\psi_k|$, we may assume $\psi_k\ge 0$.
Let $\psi_k^\star$ be the anisotropic gradient rearrangement of $\psi_k$ on $\W_R$.
By the properties of the rearrangement,
\[
\int_{\W_R}H(\nabla\psi_k^\star)^2\,dx
=
\int_{\W_R}H(\nabla\psi_k)^2\,dx,
\]
\[
\big|\{H(\nabla\psi_k^\star)\neq 0\}\big|
=
\big|\{H(\nabla\psi_k)\neq 0\}\big|,
\]
and
\[
\int_{\W_R}\psi_k\,dx\le \int_{\W_R}\psi_k^\star\,dx.
\]
Hence
\[
\mathcal F^H_\Lambda(\psi_k^\star)\le \mathcal F^H_\Lambda(\psi_k),
\]
so $\{\psi_k^\star\}$ is still a minimizing sequence.

Applying the estimate from Step 1 to $\psi_k^\star$, we get
\[
\frac14\int_{\W_R}H(\nabla\psi_k^\star)^2\,dx - C
\le \mathcal F^H_\Lambda(\psi_k^\star)\le m+1
\]
for $k$ large enough. Therefore $\{\psi_k^\star\}$ is bounded in $H_0^1(\W_R)$.
Up to a subsequence, there exists $u\in H_0^1(\W_R)$ such that
\[
\psi_k^\star \rightharpoonup u \quad\text{weakly in }H_0^1(\W_R),
\qquad
\psi_k^\star \to u \quad\text{strongly in }L^2(\W_R).
\]
Since each $\psi_k^\star$ is $H^\circ$-radial, also $u$ is $H^\circ$-radial.

By convexity of $\xi\mapsto H(\xi)^2$,
\[
\frac12\int_{\W_R}H(\nabla u)^2\,dx
\le
\liminf_{k\to\infty}
\frac12\int_{\W_R}H(\nabla\psi_k^\star)^2\,dx.
\]
Moreover, strong $L^2$ convergence implies
\[
\int_{\W_R}\psi_k^\star\,dx \to \int_{\W_R}u\,dx.
\]

It remains to deal with the penalization term.
By construction, for each $k$ the function $H(\nabla\psi_k^\star)$ is $H^\circ$-radial and radially nondecreasing.
Hence there exists $r_k\in[0,R]$ such that, up to null sets,
\[
\{H(\nabla\psi_k^\star)=0\}=\W_{r_k},
\qquad
\{H(\nabla\psi_k^\star)\neq 0\}=\W_R\setminus \W_{r_k}.
\]
Up to a further subsequence, we may assume $r_k\to r\in[0,R]$.
Then
\[
\big|\{H(\nabla\psi_k^\star)\neq 0\}\big|
=
|\W_R\setminus \W_{r_k}|
\to
|\W_R\setminus \W_r|.
\]

We claim that $H(\nabla u)=0$ a.e. in $\W_r$.
Fix $\varepsilon>0$ such that $r-\varepsilon>0$.
For $k$ sufficiently large, $r_k>r-\varepsilon$, hence
\[
H(\nabla\psi_k^\star)=0
\qquad\text{a.e. in }\W_{r-\varepsilon}.
\]
Since $H(\xi)=0$ if and only if $\xi=0$, this means $\nabla\psi_k^\star=0$ a.e. in $\W_{r-\varepsilon}$.
Therefore, for every $\Phi\in C_c^\infty(\W_{r-\varepsilon};\R^n)$,
\[
\int_{\W_R}\psi_k^\star\,{\rm div}(\Phi)\,dx
=
-\int_{\W_R}\nabla\psi_k^\star\cdot \Phi\,dx
=0.
\]
Passing to the limit, using the strong $L^2$ convergence of $\psi_k^\star$ to $u$, we obtain
\[
\int_{\W_R}u\,{\rm div}(\Phi)\,dx=0
\qquad\text{for every }\Phi\in C_c^\infty(\W_{r-\varepsilon};\R^n).
\]
Hence $\nabla u=0$ in the distributional sense in $\W_{r-\varepsilon}$, and thus a.e. in $\W_{r-\varepsilon}$.
Since $\varepsilon>0$ is arbitrary, we conclude that
\[
\nabla u=0 \qquad\text{a.e. in }\W_r,
\]
and therefore
\[
H(\nabla u)=0 \qquad\text{a.e. in }\W_r.
\]
It follows that
\[
\{H(\nabla u)\neq 0\}\subseteq \W_R\setminus \W_r
\qquad\text{up to null sets,}
\]
hence
\[
\big|\{H(\nabla u)\neq 0\}\big|
\le |\W_R\setminus \W_r|
=
\lim_{k\to\infty}\big|\{H(\nabla\psi_k^\star)\neq 0\}\big|.
\]

Collecting the previous inequalities, we get
\[
\mathcal F^H_\Lambda(u)
\le \liminf_{k\to\infty}\mathcal F^H_\Lambda(\psi_k^\star)
= m,
\]
so $u$ is a minimizer.

\medskip
\noindent
\textit{Step 3: existence of a symmetric minimizer.}
Let $v:=u^\star$ be the anisotropic gradient rearrangement of the minimizer $u$.
Then, by the rearrangement properties,
\[
\int_{\W_R}H(\nabla v)^2\,dx=\int_{\W_R}H(\nabla u)^2\,dx,
\qquad
\big|\{H(\nabla v)\neq 0\}\big|=\big|\{H(\nabla u)\neq 0\}\big|,
\]
and
\[
\int_{\W_R}u\,dx\le \int_{\W_R}v\,dx.
\]
Therefore
\[
\mathcal F^H_\Lambda(v)\le \mathcal F^H_\Lambda(u)=m.
\]
Since $u$ is a minimizer, also $v$ is a minimizer.

By construction, $v$ is nonnegative, $H^\circ$-radial, and $H(\nabla v)$ is $H^\circ$-radial and radially nondecreasing.
Hence there exists $r\in[0,R]$ such that, up to null sets,
\[
\{H(\nabla v)=0\}=\W_r,
\qquad
\{H(\nabla v)\neq 0\}=\W_R\setminus \W_r.
\]
This concludes the proof.
\end{proof}

\begin{remark}\label{rem:conditional_uniqueness}
The previous proposition gives existence of a symmetric minimizer on a Wulff ball,
which is all that is needed for the Saint--Venant inequality below. We do not
claim uniqueness here. Any uniqueness or rigidity statement would require an additional
anisotropic analogue of the Euclidean rigidity theorem used in \cite{AGNT24}, and this
goes beyond the results proved in the present paper.
 In the anisotropic
setting, one expects the corresponding extremal level sets to be homothetic Wulff
shapes, but a complete rigidity statement would require additional hypotheses and
will not be pursued here.
\end{remark}

\begin{corollary}[Saint--Venant inequality for $T_{\mathcal F}^H$]\label{cor:SV_TF_aniso}
Let $\Omega$ be a bounded open set with finite perimeter and let $\Omega^\star$ be the centered Wulff ball with $|\Omega^\star|=|\Omega|$.
Then for every $\Lambda>0$,
\[
T_{\mathcal F}^H(\Omega,\Lambda)\le T_{\mathcal F}^H(\Omega^\star,\Lambda).
\]
\end{corollary}

\begin{proof}
Fix $\psi\in H_0^1(\Omega)$, $\psi\ge 0$.
Let $\psi^\star\in H_0^1(\Omega^\star)$ be the anisotropic gradient rearrangement of $\psi$ (in the Sobolev case the boundary constant is $0$).
Then $H(\nabla\psi^\star)$ is equidistributed with $H(\nabla\psi)$, hence
\[
\int_{\Omega^\star}H(\nabla\psi^\star)^2\,dx=\int_{\Omega}H(\nabla\psi)^2\,dx,
\qquad
\big|\{H(\nabla\psi^\star)\neq 0\}\big|=\big|\{H(\nabla\psi)\neq 0\}\big|.
\]
Moreover, by the main $L^1$ comparison theorem for the anisotropic gradient rearrangement,
\[
\int_{\Omega}\psi\,dx\le \int_{\Omega^\star}\psi^\star\,dx.
\]
Therefore $\mathcal F^H_\Lambda(\psi)\ge \mathcal F^H_\Lambda(\psi^\star)$.
Taking the infimum over $\psi\in H_0^1(\Omega)$ yields the claim.
\end{proof}


\subsection{An anisotropic insulation functional and a second torsional rigidity}

Let $\Omega\subset\R^n$ be a bounded Lipschitz domain and let $m>0$. For
$\psi\in H^1(\Omega)\setminus\{0\}$ define
\begin{equation}\label{eq:G_aniso_def}
\mathcal G_m^H(\psi):=
\frac{\displaystyle \int_\Omega H(\nabla\psi)^2\,dx
+\frac1m\left(\int_{\partial\Omega}|\psi|\,d\mathcal H_H^{n-1}\right)^2}
{\left(\displaystyle \int_\Omega |\psi|\,dx\right)^2}.
\end{equation}
We then set
\begin{equation}\label{eq:TG_aniso_def}
\frac1{T_{\mathcal G}^H(\Omega,m)}
:=\inf_{\psi\in H^1(\Omega)\setminus\{0\}} \mathcal G_m^H(\psi).
\end{equation}

The functional $\mathcal G_m^H$ is related to the optimal insulation problem.
For our purposes, the relevant point is that it can be compared through the anisotropic
gradient rearrangement introduced in Section~\ref{sec:main_result}.

\medskip

The next lemma identifies the boundary term in \eqref{eq:G_aniso_def} with the singular
part of the anisotropic variation of the zero extension.

\begin{lemma}[Zero extension and anisotropic jump part]
\label{lem:zero_extension_jump}
Let $\Omega\subset\R^n$ be a bounded Lipschitz domain and let $\psi\in H^1(\Omega)$
with $\psi\ge 0$. Let $\tilde\psi$ be the extension of $\psi$ by $0$ outside $\Omega$.
Then $\tilde\psi\in BV(\R^n)$, $\tilde\psi=0$ a.e. in $\R^n\setminus\Omega$, and
\[
D^a\tilde\psi=\nabla\psi\,\mathcal L^n\llcorner\Omega,
\qquad
|D^s\tilde\psi|_H(\R^n)=\int_{\partial\Omega}\psi\,d\mathcal H_H^{n-1}.
\]
\end{lemma}

\begin{proof}
Let $\Phi\in C_c^1(\R^n;\R^n)$. Since $\tilde\psi=\psi$ in $\Omega$ and
$\tilde\psi=0$ in $\R^n\setminus\Omega$,
\[
\int_{\R^n}\tilde\psi\,\divg\Phi\,dx=\int_\Omega \psi\,\divg\Phi\,dx.
\]
Because $\Omega$ is Lipschitz and $\psi\in H^1(\Omega)$, the trace
$\tr_{\partial\Omega}\psi\in L^1(\partial\Omega)$ exists and the Gauss--Green formula gives
\[
\int_\Omega \psi\,\divg\Phi\,dx
=
-\int_\Omega \nabla\psi\cdot\Phi\,dx
+\int_{\partial\Omega}(\tr_{\partial\Omega}\psi)\,\Phi\cdot\nu_\Omega\,d\mathcal H^{n-1}.
\]
Hence, in the sense of vector measures,
\[
D\tilde\psi
=
\nabla\psi\,\mathcal L^n\llcorner\Omega
-(\tr_{\partial\Omega}\psi)\,\nu_\Omega\,\mathcal H^{n-1}\llcorner\partial\Omega.
\]
Therefore
\[
D^a\tilde\psi=\nabla\psi\,\mathcal L^n\llcorner\Omega,
\qquad
D^s\tilde\psi
=-(\tr_{\partial\Omega}\psi)\,\nu_\Omega\,\mathcal H^{n-1}\llcorner\partial\Omega.
\]
Since $\psi\ge 0$, we obtain
\[
|D^s\tilde\psi|_H(\R^n)
=
\int_{\partial\Omega}H(\nu_\Omega)\,\psi\,d\mathcal H^{n-1}
=
\int_{\partial\Omega}\psi\,d\mathcal H_H^{n-1},
\]
as claimed.
\end{proof}

\medskip

The next proposition we prove the three main results  needed in the proof of the
Saint--Venant inequality: preservation of the Dirichlet term, preservation of the
boundary term, and improvement of the $L^1$ denominator.

\begin{proposition}[Rearranged competitors for $\mathcal G_m^H$]
\label{prop:rearranged_competitor_G}
Let $\Omega\subset\R^n$ be a bounded Lipschitz domain and let $\Omega^\star$ be the
centered Wulff ball with $|\Omega^\star|=|\Omega|$. Let $\psi\in H^1(\Omega)$, $\psi\ge 0$,
let $\tilde\psi$ be its zero extension, and let $\psi^\star$ be the restriction to
$\Omega^\star$ of the anisotropic gradient rearrangement of $\tilde\psi$.
Then $\psi^\star\in H^1(\Omega^\star)$ and
\begin{align}
\int_{\Omega^\star} H(\nabla\psi^\star)^2\,dx
&=
\int_\Omega H(\nabla\psi)^2\,dx,
\label{eq:G_prop_energy}
\\[1mm]
\int_{\partial\Omega^\star}\psi^\star\,d\mathcal H_H^{n-1}
&=
\int_{\partial\Omega}\psi\,d\mathcal H_H^{n-1},
\label{eq:G_prop_boundary}
\\[1mm]
\int_\Omega \psi\,dx
&\le
\int_{\Omega^\star}\psi^\star\,dx.
\label{eq:G_prop_L1}
\end{align}
\end{proposition}

\begin{proof}
By Lemma~\ref{lem:zero_extension_jump}, $\tilde\psi\in BV(\R^n)$, $\tilde\psi=0$ a.e. in
$\R^n\setminus\Omega$, and
\[
D^a\tilde\psi=\nabla\psi\,\mathcal L^n\llcorner\Omega,
\qquad
|D^s\tilde\psi|_H(\R^n)=\int_{\partial\Omega}\psi\,d\mathcal H_H^{n-1}.
\]
Let $\psi^\star$ be the anisotropic gradient rearrangement of $\tilde\psi$ on $\Omega^\star$.
By construction, $H(\nabla\psi^\star)$ is equidistributed with
$H(\nabla\psi)\chi_\Omega$. 
We have that 
\[
H(\xi)\ge \alpha|\xi|
\qquad\text{for every }\xi\in\R^n .
\]
Hence
\[
\|\nabla\psi^\star\|_{L^2(\Omega^\star)}
\le \alpha^{-1}\|H(\nabla\psi^\star)\|_{L^2(\Omega^\star)}<\infty.
\]
Since also \(\psi^\star\in L^2(\Omega^\star)\), we conclude that
\[
\psi^\star\in H^1(\Omega^\star).
\]
Moreover, equidistribution gives \eqref{eq:G_prop_energy}.

By the definition of the rearrangement in the BV setting, $\psi^\star$ is
$H^\circ$-radial and constant on $\partial\Omega^\star$, say
\[
\psi^\star=c
\qquad \text{on }\partial\Omega^\star,
\]
with
\[
c=\frac{|D^s\tilde\psi|_H(\R^n)}{P_H(\Omega^\star)}.
\]
Therefore
\[
\int_{\partial\Omega^\star}\psi^\star\,d\mathcal H_H^{n-1}
=
c\,P_H(\Omega^\star)
=
|D^s\tilde\psi|_H(\R^n)
=
\int_{\partial\Omega}\psi\,d\mathcal H_H^{n-1},
\]
which proves \eqref{eq:G_prop_boundary}.

Finally, \eqref{eq:G_prop_L1} follows from Theorem~\ref{thm:main_comparison} applied to
$\tilde\psi$, namely
\[
\int_\Omega \psi\,dx
=
\int_{\R^n}\tilde\psi\,dx
\le
\int_{\Omega^\star}\psi^\star\,dx.
\]
\end{proof}

\begin{corollary}[Saint--Venant inequality for $T_{\mathcal G}^H$]
\label{cor:SV_TG_aniso}
Let $\Omega\subset\R^n$ be a bounded Lipschitz domain and let $\Omega^\star$ be the
centered Wulff ball with $|\Omega^\star|=|\Omega|$. Then for every $m>0$,
\[
T_{\mathcal G}^H(\Omega,m)\le T_{\mathcal G}^H(\Omega^\star,m).
\]
\end{corollary}

\begin{proof}
Let $\psi\in H^1(\Omega)\setminus\{0\}$. Replacing $\psi$ with $|\psi|$, we may assume
that $\psi\ge 0$, because $|\psi|\in H^1(\Omega)$,
\[
H(\nabla|\psi|)=H(\nabla\psi)\qquad\text{a.e. in }\Omega,
\]
and the other terms in \eqref{eq:G_aniso_def} are unchanged.

Let $\psi^\star\in H^1(\Omega^\star)$ be given by
Proposition~\ref{prop:rearranged_competitor_G}. By
\eqref{eq:G_prop_energy}--\eqref{eq:G_prop_L1}, we have
\[
\mathcal G_m^H(\psi)\ge \mathcal G_m^H(\psi^\star).
\]
Taking the infimum over $\psi\in H^1(\Omega)\setminus\{0\}$ yields
\[
\frac1{T_{\mathcal G}^H(\Omega,m)}
\ge
\frac1{T_{\mathcal G}^H(\Omega^\star,m)},
\]
which is equivalent to the conclusion.
\end{proof}

\section*{Acknowledgments}
The first author G.P. is partially supported by Gruppo Nazionale per l’Analisi Matematica, la Probabilità e le loro Applicazioni
(GNAMPA) of Istituto Nazionale di Alta Matematica (INdAM) and  by INdAM - GNAMPA Project "Disuguaglianze funzionali di tipo Geometrico e Spettrale", CUP E5324001950001.

 The second author Y.Y. is supported by the Program of China Scholarship Council (Grant No. 202506030179).

\bibliographystyle{plain}
\bibliography{ref}

\Addresses

\end{document}